\documentclass{article}
\usepackage[utf8]{inputenc}
\usepackage{mathtools}
\usepackage{amsmath}
\usepackage{amsthm}
\usepackage{amssymb}
\usepackage{comment}
\usepackage{dirtytalk}
\usepackage{amsfonts}

\usepackage{color}
\usepackage{cite}

\makeatletter
\theoremstyle{plain}

\numberwithin{equation}{section}
\newtheorem{theorem}{Theorem}[section]

\newcommand{\rc}{\begin{color}{red}}
\newcommand{\ec}{\end{color}}

\newtheorem{condition}[theorem]{Condition}

\newtheorem{corollary}[theorem]{Corollary}

\newtheorem{definition}[theorem]{Definition}
\newtheorem{example}[theorem]{Example}

\newtheorem{lemma}[theorem]{Lemma}

\newtheorem{proposition}[theorem]{Proposition}
\newtheorem{remark}[theorem]{Remark}

\newcommand{\R}{{\mathbb{R}}}

\makeatother

\begin{document}
\title{Stochastic dynamics  on evolving geometric graphs}
\author{Alexei Daletskii\thanks{Department of Mathematics, The University of York, York YO10\;5DD, U.K. ({\tt alex.daletskii@york.ac.uk}).} \and Dmitri Finkelshtein\thanks{Department of Mathematics,
Swansea University, Singleton Park, Swansea SA2 8PP, U.K. ({\tt d.l.finkelshtein@swansea.ac.uk}).}}

\maketitle

\begin{abstract}
We consider an infinite locally finite system (configuration) $\gamma$
of particles distributed over a Euclidean space $X$. Each particle
located at $x\in X$carries an internal parameter (mark, or ``spin'')
$\sigma_{x}\in S=\mathbb{R}.$ Such collections of particles form
the space of marked configurations $\Gamma(X,S)$. We construct the
following stochastic dynamics in $\Gamma(X,S)$: while the configuration
$\gamma$ of particle positions performs a random evolution, the corresponding
marks interact with each other and perform a coupled infinite-dimensional
diffusion. The study of a spin dynamics on a fixed configuration $\gamma$
was initiated in \cite{DaF} and continued in \cite{ChDa}, and is
based on the generalisation of the Ovsjannikov method. In the present
paper, the underlying configuration evolves according to a Birth-and-Death
process. We prove the existence and uniqueness of such dynamics and show that it forms a càdlàg process in $\Gamma(X,S)$.

\end{abstract}

\tableofcontents

\section{Introduction}

We consider the following model of interacting particle systems. Our
system consists of an infinite number of particles of physical, biological or any other nature, which are located in a Euclidean
space $X$ and form a locally finite set (configuration) $\gamma\subset X$.
The space of all such configurations will be denoted by $\Gamma(X)$.
Each particle located at $x\in X$ carries an internal parameter (mark,
or ``spin'') $\sigma_{x}\in S=\mathbb{R}.$ The corresponding marks
interact with each other and perform a random evolution $\Xi_t$, described
by an infinite system of coupled diffusion equations, with a finite
radius of interaction.

For a fixed configuration $\gamma$, the evolution $\Xi_t(\gamma)=(\xi_{x,t})_{x\in\gamma}$
can be studied in the product space $S^{\gamma}$, or, rather, a scale
of its Hilbert or Banach subspaces, see \cite{DaF}, \cite{Dal},
\cite{ADGC} and \cite{ChDa}. In the present paper, however, we consider a
more complex situation where the configuration $\gamma$ evolves,
too, according to a given stochastic process $\gamma_{t}$, $t\in[0,T].$
Therefore, the product space, in which the spin dynamics lives, depends
on time $t\in[0,T]$ and realisation of process $\gamma_{t}$.

A natural
choice for a  space where the combined dynamics $(\gamma_{t},\Xi_{t}(\gamma_t))$,
$t\in[0,T],$ can be studied, is the space $\Gamma(X,S)$
 of marked
configurations
$$\widehat{\gamma}=\{(x,\sigma_x),x\in\gamma\in\Gamma(X),\sigma_x\in S\}.$$
Indeed,  $\Gamma(X,S)$ possesses the structure of a fibre bundle
\begin{equation}\label{eq:fibre}
p_X:\Gamma(X,S)\rightarrow \Gamma(X),\,p_X^{-1}(\gamma)\simeq S^\gamma,
\end{equation}
see Section \ref{sec-conf} for a
detailed description of those structures.

In this paper, we consider the case where $\gamma_{t},$ $t\in[0,T]$,
is a $\Gamma(X)$-valued jump process, satisfying certain additional conditions, see Section \ref{sec:jump_proc},
where the class ${\mathcal{R}}({\mathrm {L}}\Gamma(X))$ of these processes is introduced,
and Section \ref{Sec DoC} for main examples given by Birth-and-Death processes.
The dynamics of marks is defined by the following (heuristic) system of stochastic differential equations in $S$:
\begin{align}
    d\xi_{x,t}=&\left[\phi_{x}(\xi_{x,t})+\sum_{y\in\gamma_{x,t}}\phi_{xy}(\xi_{x,t},\xi_{y,t})\right]dt+\nonumber\\
    &+\left[\sum_{y\in\gamma_{x,t}}\psi_{xy}(\xi_{x,t},\xi_{y,t})\right]dW_{x,t},\,x\in\gamma_t,\label{eq:SDE}
\end{align}
where $\left(W_{x,t}\right)_{x\in\gamma}$
is a family of independent Wiener processes and functions $\phi_x:S\rightarrow S$, $\phi_{xy}:S\times S\rightarrow S$ and $\psi_{xy}:S\times S\rightarrow S$
satisfy certain regularity conditions, see Section \ref{sec:phantom_dynamics}.

A rigorous interpretation of system \eqref{eq:SDE} can be obtained using the fibration structure \eqref{eq:fibre}.
The main result of this paper, still formulated in a somewhat heuristic form, is the following statement
(cf. Theorem \ref{theor:solution} for its rigorous formulation).
\begin{theorem}\label{theor:sol0}
For any $(\gamma_t)_{t\in [0,T]}\in\mathcal{R}(\mathrm{L}\Gamma(X))$ and $\widehat{\gamma}_0\in \Gamma(X,S) $,
there exists a unique $\Gamma(X,S)$-valued process $(\widehat{\gamma}_{t})_{t\in[0,T]}$
with initial value   $\widehat{\gamma}_0$ and such that, under identification  \eqref{eq:fibre},
\begin{equation}
 \widehat{\gamma}_t=(\gamma_t,\Xi_t), \, \, t\in[0,T],
\end{equation}
where $\Xi_t=(\xi_{x,t})_{x\in\gamma}$ is a solution of system \eqref{eq:SDE} along $\gamma_{t}$. The process $(\widehat{\gamma}_{t})_{t\in[0,T]}$
is càdlàg in the topology of  $\Gamma(X,S)$.
\end{theorem}

The structure of the paper is as follows. Section \ref{sec-conf} is devoted to a brief review of the main structures on marked configuration spaces.
In Section \ref{sec:jump_proc}, we introduce a class of random processes on $\Gamma(X)$ satisfying a priori conditions required for what follows.
Section \ref{sec:phantom_dynamics} is devoted to the discussion of a rigorous form of system \eqref{eq:SDE} on a fixed sample path of the underlying process $\gamma_t$.
Our main result - Theorem \ref{theor:sol0} / \ref{theor:solution} is proved in Section \ref{sec:dynamics}.
In Section \ref{Sec DoC}, we discuss a class of Birth-and-Death processes satisfying conditions of Section \ref{sec:jump_proc}.
Finally, in the Appendix, we provide some auxiliary technical results based on our previous work on infinite SDE systems.

\section{Marked configuration spaces $\Gamma(X,S)$}\label{sec-conf}

As a location (phase) space $X$ for our particle system, let us fix
the $d$-dimensional ($d\geq1$) Euclidean space $\mathbb{R}^{d}$.
It is endowed with the Lebesgue measure $\mathrm{d}x$ on the Borel
$\sigma$-algebra $\mathcal{B}(X)$. By $\mathcal{B}_{0}(X)$ we denote
the ring of all bounded sets from $\mathcal{B}(X)$. The configuration
space $\Gamma(X)$ consists of all locally finite subsets of $X$,
that is,
\begin{equation}
\Gamma(X)=\left\{ \gamma\subset X:\ {\cal N}\left(\gamma_{\Lambda}\right)<\infty\text{ for any }\Lambda\in\mathcal{B}_{0}(X)\right\} ,\label{gamma}
\end{equation}
where ${\cal N}\left(\gamma_{\Lambda}\right)$ stands for the cardinality
of the restriction $\gamma_{\Lambda}:=\gamma\cap\Lambda$. Let $C_{0}(X)$
be the set of all continuous functions $f:X\rightarrow\mathbb{R}$
with compact support. The space $\Gamma(X)$ is equipped in the standard
way with the \emph{vague topology}, which is the weakest one that
makes continuous all maps
\[
\Gamma(X)\ni\gamma\mapsto\left\langle f,\gamma\right\rangle :=\sum_{x\in\gamma}f(x),\quad f\in C_{0}(X).
\]
It is well known (see, e.g., \cite[Section 15.7.7]{Kal}) that $\Gamma(X)$
is a Polish (i.e., separable completely metrizable) space in this
topology; an explicit construction of the appropriate metric can be
found in \cite{KKut}. By $\mathcal{P}(\Gamma(X))$ we denote the
space of all probability measures on the corresponding Borel $\sigma$-algebra
$\mathcal{B}(\Gamma(X))$.

Let now $S$ be another Euclidean space $\mathbb{R}^{m}$ (with $m\neq d$
in general) and consider the Cartesian product $X\times S$.
For any element $\widehat{x}:=(x,s)$ of $\widehat{X}$ its $S$-component
$s$ may be seen as a mark (spin, charge, etc.) attached to a particle
placed at position $x\in X$. Given a set $\Lambda\subset X$, we
will often write for short $\widehat{\Lambda}:=\Lambda\times S$.
The canonical projection $p_{X}:X\times S\rightarrow X$ can be naturally
extended to the configuration space $\Gamma(X\times S)$.
Observe that for a configuration $\widehat{\gamma}\in\Gamma(X,S)$
its image $p_{X}(\widehat{\gamma})$ is a subset of $X$ that possibly
admits accumulation and multiple points, and hence does not in general
belong to $\Gamma(X)$. The\emph{\ marked} configuration space\emph{\ }$\Gamma(X,S)$
is then defined in the following way (see e.g. \cite{CKMS}, \cite{DVJ1},
\cite{KMM}):
\begin{equation}
\Gamma(X,S):=\left\{ \widehat{\gamma}\in\Gamma(X\times S):\text{ }p_{X}(\widehat{\gamma})\in\Gamma(X)\right\} .\label{def-marked}
\end{equation}
We will systematically use the notation
\[
\gamma_{\Lambda}:=\gamma\cap\Lambda\text{ and }\widehat{\gamma}_{\Lambda}:=\widehat{\gamma}\cap\widehat{\Lambda}
\]
for $\gamma\in\Gamma(X),\ \widehat{\gamma}\in\Gamma(X,S),\Lambda\subset X$
and cylinder sets $\widehat{\Lambda}:=\Lambda\times S$.

We equip $\Gamma(X,S)$ with the so-called $\tau$-topology defined
as the weakest one that makes continuous the map
\begin{equation}
\Gamma(X,S)\ni\widehat{\gamma}\mapsto\left\langle g,\widehat{\gamma}\right\rangle :=\sum_{(x,s)\in\widehat{\gamma}}g(x,s)\label{top}
\end{equation}
for any bounded continuous function $g:X\times S\rightarrow\mathbb{R}$
with $\mathrm{supp\,}g\subset\Lambda\times S$ for some $\Lambda\in\mathcal{B}_{0}(X)$,
i.e. with spatially compact support. This topology has been employed
in different frameworks in e.g.\textbf{\ }\cite{AKLU}, \cite{DKKP2}
and\ \cite{KunaPhD}; for a short account of its properties, see also
\cite{DKKP}. It makes
$\Gamma(X,S)$ a Polish space. For an example of the $\tau$-consistent
metric on $\Gamma(X,S)$ see Section 2 of \cite{CG}. We then endow
$\Gamma(X,S)$ with the associated Borel $\sigma$-algebra $\mathcal{B}(\Gamma(X,S))$.
This is the smallest $\sigma$-algebra for which the counting variable
\begin{equation}
\widehat{\gamma}\mapsto N(\widehat{\gamma}\cap\Delta)\label{ng}
\end{equation}
is measurable for any $\Delta\in\mathcal{B}(X\times S)$ with $p_{X}(\Delta)\in\mathcal{B}_{0}(X)$.
\begin{remark}
\label{fibre}

The projection map
$$p_X:\Gamma(X,S)\rightarrow\Gamma(X)$$
induces a fibre bundle-type structure
on $\Gamma(X,S)$. It follows directly from the definition of the corresponding topologies
that the map $p_{X}:\Gamma(X,S)\rightarrow\Gamma(X)$ is continuous.
Hence for any configuration $\gamma$ the fibre $p_{X}^{-1}(\gamma)$
is a Borel subset of $\Gamma(X,S)$. Moreover, any $\widehat{\gamma}\in{\Gamma}(X,S)$ can be uniquely represented
by the pair
\[
\widehat{\gamma}=(\gamma,\sigma_{\gamma}),\text{ where }\gamma=p_{X}(\widehat{\gamma})\in\Gamma(X) \text{ and }\sigma_{\gamma}=(\sigma_{x})_{x\in\gamma}\in S^{\gamma},
\]
and $S^{\gamma}:=\prod\limits_{x\in\gamma}S_{x},\ S_{x}:=S.$ So the fibres $p_{X}^{-1}(\gamma)$ are isomorphic to the product-spaces $S^{\gamma}$, and we will denote by
 $i_\gamma  $
 the corresponding isomorphism, so that
 \begin{equation}\label{eq:fibre1}
     i_\gamma : p_{X}^{-1}(\gamma)\rightarrow S^{\gamma},\;i_\gamma (\widehat{\gamma})=(\sigma_{x})_{x\in\gamma}\in S^{\gamma}.
 \end{equation}
\end{remark}

\section{A class of jump processes on $\Gamma(X)$ and their phantoms }\label{sec:jump_proc}

Consider a suitable filtered and complete probability space $\mathbf{P}:=(\Omega,\mathcal{F},\mathbb{P},\mathbb{F})$
and let $\gamma_{t}\in\Gamma(X)$, $t\in[0,T]$, be a $\Gamma(X)$-valued
random process on $\mathbf{P}$. Define the set
\[
\gamma^{(T)}:=\bigcup_{t\in[0,T]}\gamma_{t}\subset X,
\]
which will be called the phantom of process $\gamma_{t}$. We assume
that the following condition holds.

\begin{condition} \label{cond-fin}For any compact $\Lambda\subset X$
the mapping
\[
[0,T]\ni t\mapsto\gamma_{t}\cap\Lambda
\]
changes its value only a finite number of times, a.s. \end{condition}

We introduce the notation

\[
n_{x,R}(\gamma):=N(\gamma\cap B_{x,R}),
\]
where $B_{x,R}$ is the ball of radius $R$ centered at $x\in X$.

\begin{definition}\label{def:log}
    Define the space ${\mathrm L}\Gamma(X)$ of configurations $\gamma\in\Gamma(X)$ satisfying the following estimate: for any $ R>0$ there exists a  constant $a_{R}=a_R(\gamma)$ s.t.
\begin{equation}\label{eq:log}
     n_{x,R}(\gamma)\leq a_{R}(\gamma)(1+\log(1+\left\vert x\right\vert )),
\end{equation}
for all $x\in X$.
\end{definition}

\begin{remark}\label{rem:Poisson}
    Estimate \eqref{eq:log} holds for a typical configuration $\gamma$ distributed according to a Poisson or, more generally, Gibbs measure on $\Gamma(X)$  with a superstable low regular interaction energy, in which case $n_{x,R}$ is proportional to $\sqrt{1+\log(1+\left\vert x\right\vert )}$, see e.g. \cite{Ruelle} and \cite{KlYa} (p. 1047).
\end{remark}

We assume that the process $\gamma_{t}$, $t\in[0,T]$, satisfies
the following condition.

\begin{condition} \label{cond-log-1}
For a.a. realizations of the process $\gamma_t,\,t\in [0,T]$, we have
\begin{equation}\label{eq:log}
    \gamma^T\in {\mathrm L}\Gamma(X).
\end{equation}
\end{condition}
Observe that the condition above implies that
\begin{equation}\label{eq:log}
    \gamma^t\in {\mathrm L}\Gamma(X)\, \text{for all }  t\in[0,T].
\end{equation}

\begin{definition}
The class of processes satisfying Conditions \ref{cond-fin} and \ref{cond-log-1} will be denoted by ${\mathcal{R}}({\mathrm {L}}\Gamma(X))$.
\end{definition}
In Section \ref{Sec DoC}, we will discuss an example of such processes, given by so-called Birth-and-Death (BaD) processes.

\
\section{Dynamics in the Phantom Space}\label{sec:phantom_dynamics}

In this section, we fix a realisation of the process $\gamma_t,\,t\in\mathcal{T}$,
such that the bound \eqref{eq:log} holds, and the corresponding phantom configuration $\gamma^T$.
Consider the product space
\[
\mathbb{S}^{T}:=S^{\gamma^{T}},
\]
which we will call the Phantom Space. We will construct a dynamics
on this space, defined by the pair interaction between the spins of elements
of configurations $\gamma_{t}$ and the corresponding single particle
potentials. This dynamics will solve the system of non-autonomous
stochastic equations
\begin{align}
\xi_{x,t} & =\zeta_{x}+\int_{0}^{t}\Phi_{x}(\Xi_{s},s)ds+\int_{0}^{t}\Psi_{x}(\Xi_{s},s)dW_{x,s},\ x\in\gamma^{T},\ \label{MainSystem}\\
t & \in\mathcal{T}:=[0,T],\ T>0,\nonumber
\end{align}
where $\Xi_{t}=\left(\xi_{x,t}\right)_{x\in\gamma^{T}}$ and $\left(W_{x,t}\right)_{x\in\gamma^{T}}$
is a family of independent Wiener processes on the probability space
$\mathbf{P}$, and $\zeta_{x}\in S$.
Thus, we aim to find a strong solution of SDE system \eqref{MainSystem}, that is, a family $\Xi_{t}=\left(\xi_{x,t}\right)_{x\in\gamma^{T}}$ of continuous adapted stochastic
processes on P such that the equality \eqref{MainSystem} holds for all $t \in\mathcal{T}$, almost surely, that is, on a common for all $t$ set of probability $1$.

The coefficients defined explicitly
in Assumption \ref{mainass} below.

First we need to introduce the following notations. We fix $\rho>0$
and denote by $n_{x}(\gamma):=n_{x,\rho}(\gamma)$, $x\in\gamma$,
the number of elements in the set
\[
b_{x}:=\left\{ y\in\gamma:\left\vert x-y\right\vert \leq\rho\right\} .
\]
Observe that $n_{x}(\gamma)\geq1$ for all $x\in\gamma$, because
$x\in b_{x}$. We will also use the notation $\gamma_{x}:=b_{x}\setminus\left\{ x\right\} \equiv\left\{ y\in\gamma:\left\vert x-y\right\vert \leq\rho,y\neq x\right\} $
and write $\gamma\setminus x$ instead of $\gamma\setminus\left\{ x\right\} $.

Let us consider measurable functions $\phi:S\rightarrow\mathbb{R}$,
$\varphi_{xy}:S^{2}\rightarrow\mathbb{R}^{1}$ and $\psi_{xy}:S^{2}\rightarrow\mathbb{R}^{1}$,
$x,y\in\gamma^{T}$, satisfying the following condition.

\begin{condition} \label{cond-diss}\label{mainass}
\begin{enumerate}
\item Functions $\varphi_{xy}$ satisfy uniform Lipschitz condition
\begin{align*}
\left\vert \varphi_{xy}(\sigma_{1},s_{1})-\varphi_{xy}(\sigma_{2},s_{2})\right\vert  & \leq\bar{a}\left(\left\vert \sigma_{1}-\sigma_{2}\right\vert +\left\vert s_{1}-s_{2}\right\vert \right),\\
\left\vert \varphi_{xy}(\sigma_{1},s_{1}))\right\vert  & \leq\bar{a}\left(1+\left\vert \sigma_{1}\right\vert +\left\vert s_{1}\right\vert \right),
\end{align*}
for some constant $\bar{a}>0$ and all $x,y\in\gamma,\ \sigma_{1},\sigma_{2},s_{1},s_{2}\in S.$
\item there exist constants $c>0,R\ge2$ such that
\begin{equation}
|\phi(\sigma)|\leq c(1+|\sigma|^{R}),\ \sigma\in S;\label{bound1}
\end{equation}
\item there exists $b>0$ such that
\begin{equation}
(\sigma_{1}-\sigma_{2})(\phi(\sigma_{1})-\phi(\sigma_{2})\leq b(\sigma_{1}-\sigma_{2})_{.}^{2},\ \sigma_{1},\sigma_{2}\in S;\label{bound-diss}
\end{equation}
\item there exists constant $M>0$ such that
\begin{equation}
\left\vert \psi_{xy}(\sigma_{1},\xi_{1})-\psi_{xy}(\sigma_{2},\xi_{2})\right\vert \leq M\left(\left\vert \sigma_{1}-\sigma_{2}\right\vert +\left\vert \xi_{1}-\xi_{2}\right\vert \right),\label{cond-lip1}
\end{equation}
for $x,y\in\gamma,\ \sigma_{1},\sigma_{2},s_{1},s_{2}\in\mathbb{R}.$
\end{enumerate}
\end{condition}

Suppose that the drift and diffusion coefficients $\Phi_x$ and $\Psi_x$ have, respectively, the form

\begin{equation}
\Phi_{x}((z_{y})_{y\in\gamma},t):=\phi_{x}(z_{x})+\sum_{y\in\gamma_{x,t}}\phi_{xy}(z_{x},z_{y}),\text{for all }x\in\gamma_{t},\label{Phi-form1}
\end{equation}
and
\begin{equation}
\Psi_{x}((z_{y})_{y\in\gamma},t):=\sum_{y\in\gamma_{x,t}}\psi_{xy}(z_{x},z_{y})\ \text{for all }x\in\gamma_{t},\label{Psi-form1}
\end{equation}
and set $\Phi_{x}((z_{y})_{y\in\gamma},t)=\Psi_{x}((z_{y})_{y\in\gamma},t)=0$
for $x\not\in\gamma_{t}$.
\begin{lemma}
\label{DriftLemma} Suppose that $\sigma_{1},\sigma_{2}\in\mathbb{R}$
and $Z_{1},Z_{2}\in\mathbb{S}^{T}$. Then for all $x\in\gamma^{T}$
we have the following inequalities:
\begin{align*}
|\Psi_{x}(Z_{1},t)-\Psi_{x}(Z_{2},t)| & \leq M(n_{x,t}+1)|z_{1,x}-z_{2,x}|+M\sum_{y\in\gamma_{x}}|z_{1,y}-z_{2,y}|,\\[0.01in]
|\Psi_{x}(0,t)| & \leq Mn_{x,t},
\end{align*}
and
\[
|\Phi_{x}(Z_{1},t)|\leq c(1+|z_{1,x}|^{R})+\bar{a}n_{x,t}(1+2|z_{1,x}|)+\bar{a}\sum_{y\in\gamma_{x}}|z_{1,y}|,
\]
\begin{multline*}
(z_{1,x}-z_{2,x})(\Phi_{x}(Z_{1},t)-\Phi_{x}(Z_{2},t))\\
\leq(b+\frac{1}{2}+4\bar{a}^{2}n_{x,t}^{2})(z_{1,x}-z_{2,x})^{2}+\frac{1}{2}\bar{a}^{2}n_{x,t}\sum_{y\in\gamma_{x}}(z_{1,y}-z_{2,y})^{2},
\end{multline*}
where constants $M,c,b$ and $\bar{a}$ are defined in Assumption
\ref{mainass}.
\end{lemma}
\begin{proof}
The proof follows by a direct calculation using assumptions on $\Phi$
and $\Psi$.
\end{proof}
\smallskip{}

The specific form of the coefficients requires the development of a special
framework. Indeed, we will be looking for a solution of (\ref{MainSystem})
in a scale of expanding Banach spaces of weighted sequences, which
we introduce below.

We start with a general definition and consider a family $\mathfrak{B}$
of Banach spaces $B_{\alpha}$ indexed by $\alpha\in\mathcal{A}:=[\alpha_{\ast},\alpha^{\ast})$
with fixed $0\leq\alpha_{\ast},\alpha^{\ast}<\infty$, and denote
by $\left\Vert \cdot\right\Vert _{B_{\alpha}}$ the corresponding
norms. When speaking of these spaces and related objects, we will
always assume that the range of indices is $[\alpha_{\ast},\alpha^{\ast})$,
unless stated otherwise. The interval $\mathcal{A}$ remains fixed
for the rest of this work. We will also use the corresponding closed
interval $\mathcal{\bar{A}}:=[\alpha_{\ast},\alpha^{\ast}].$
\begin{definition}
\label{Defscale}The family $\mathfrak{B}$ is called a scale if
\[
B_{\alpha}\subset B_{\beta}\ {\text{and }}\left\Vert u\right\Vert _{B_{\beta}}\leq\left\Vert u\right\Vert _{B_{\alpha}}{\text{ for any }}\alpha<\beta,\ u\in B_{\alpha},\ \alpha,\beta\in\mathcal{\bar{A}},
\]
where the embedding means that $B_{\alpha}$ is a dense vector subspace
of $B_{\beta}$.
\end{definition}
We will use the following notations:
\[
{\overline{B}}:={\bigcup}_{\alpha\in\left[\alpha_{\ast},\alpha^{\ast}\right)}B_{{\alpha}},\quad\underline{B}:={\bigcap}_{\alpha\in\left(\alpha_{\ast},\alpha^{\ast}\right]}B_{{\alpha}}.
\]
\bigskip{}
We also introduce the spaces
\[
B_{\alpha+}:=\cap_{\beta>\alpha}B_{\beta}\ \text{and }B_{\beta-}:=\cup_{a<\beta}B_{\alpha}.
\]

The main scale we will be working with is given by the following spaces
of weighted sequences.
\begin{itemize}
\item \label{DefSequenceSpaceScale} For all $p>0$ and $\alpha\in\mathcal{A}$
let
\begin{align*}
l_{\mathfrak{a}}^{p} & :=\left\{ z\in\mathbb{R}^{\gamma}\ \left\vert \Vert z\Vert_{l_{\alpha}^{p}}:=\left(\sum_{x\in\gamma}w(x)^{-1}\mid z_{x}\mid^{p}\right)^{1/p}\le\infty\right.\right\} ,\ w(x)=e^{\alpha\mid x\mid},\\[1em]
\mathcal{L}^{p} & :=\{l_{\mathfrak{a}}^{p}\}_{\mathfrak{a}\in\mathcal{A}}
\end{align*}
be, respectively, a Banach space of weighted real sequences and a
scale of such spaces.
\item \label{Zspaces} For all $p>0$ and $\mathfrak{a}\in\mathcal{A}$
let $\mathcal{R}_{\mathfrak{a}}^{p}$ denote the Banach space of $l_{\mathfrak{a}}^{p}$-valued
adapted processes $\xi$ with finite norm
\[
\Vert\xi\Vert_{\mathcal{R}_{\mathfrak{a}}^{p}}:=\left(\sup\left\{ \mathbb{E}\Vert\xi\Vert_{l_{\mathfrak{a}}^{p}}^{p}\ \left\vert \ t\in\mathcal{T}\right.\right\} \right)^{\frac{1}{p}}<\infty,
\]
and let be $\mathcal{R}^{p}:=\{\mathcal{R}_{\mathfrak{a}}^{p}\}_{\mathfrak{a}\in\mathcal{A}}$
a scale of such spaces.
\end{itemize}
The main result of this section is the following theorem.
\begin{theorem}
\label{Existence} Suppose that Assumption \ref{cond-diss} holds.
Then, for all $p\geq R$ and any $\mathcal{F}_{0}$-measurable initial
condition $\bar{\zeta}:=(\zeta_{x})_{x\in\gamma^{T}}\in L_{\alpha}^{p}(\gamma^{T})$,
$\alpha\in\mathcal{A}$, stochastic system (\ref{MainSystem}) admits
a unique (up to indistinguishibility) strong solution $\Xi^T\in\mathcal{R}_{\alpha+}^{p}(\gamma^{T})$.
It satisfies the growth estimate
\begin{equation}
\mathbb{E}\left\Vert \Xi^T_{t}\right\Vert _{l_{\beta}^{p}(\gamma^{T})}^{p}\leq C_{1}K_{T}(\alpha,\beta)\left(\mathbb{E}\left\Vert \Xi^T_{0}\right\Vert _{l_{\alpha}^{p}(\gamma^{T})}^{p}+\left\Vert \bar{C}_{2}\right\Vert _{l_{\alpha}^{p}(\gamma^{T})}\right),\label{ineq222-1-1-1}
\end{equation}
where $K_{T}(\alpha,\beta)$ is defined by formulae (\ref{eq:K})
and (\ref{eq:L}) in the Appendix. Moreover, the map
\[
L_{\alpha}^{p}(\gamma^{T})\ni\bar{\zeta}\mapsto\Xi^T\in\mathcal{R}_{\beta}^{p}(\gamma^{T})
\]
is continuous for any $\beta>\alpha$.
\end{theorem}
\begin{remark}
Assumption $p\geq R$ ensures that given $\xi\in\mathcal{R}_{\beta}^{p}$
the random variable $\phi(\xi_{t}),t\geq0$, is integrable.
\end{remark}
\begin{proof}
In the case of the drift and diffusion coefficients $\Phi_{x}$ and
$\Psi_{x}$ formed using time-independent functions $a,\phi$ and
$\psi$ instead of $\hat{a},\hat{\phi}$ and $\hat{\psi}$, the result
was proved in \cite{ChDa}. Observe that the functions $\hat{a},\hat{\phi}$
and $\hat{\psi}$ satisfy the bounds of Assumption \ref{mainass}
with the same universal constants $\bar{a},b,c$ and $M$. Therefore
the proof can be extended to this case in a straightforward manner.
We present here only its sketch.

\textbf{Step 1.} Consider a sequence of processes $\left\{ \Xi_{t}^{n}\right\} _{n\in\mathbb{N}}$,
$t\in\mathcal{T}$, that solve finite cut-offs of system (\ref{MainSystem}),
obtained by setting $\xi_{x,t}=\zeta_{x}$ for $x\in\gamma^{T}$ such
that $\left\vert x\right\vert >n$. A more detailed description of
the construction of that sequence is given in Appendix.

\textbf{Step 2.} Prove a uniform bound of the family $\left\{ \Xi_{t}^{n}\right\} _{n\in\mathbb{N}}$
in $\mathcal{R}_{\beta}^{p}$ for any $\beta>\alpha$. This is a difficult
step requiring use of our version of the comparison theorem and Gronwall-type
inequality in the scale of spaces, which is in turn based on the Ovsjannikov
method, see \cite{ChDa}.

\textbf{Step 3.} The uniform bound above implies convergence of the sequence
$\Xi^{n}$, $n\rightarrow\infty$, to a process $\Xi^T\in\mathcal{R}_{\beta}^{p}$,
$\beta>\alpha$. Our next goal is to prove that this process solves
system (\ref{MainSystem}). The multiplicative noise term does not
allow to achieve this by a direct limit transition. To proceed, we
construct an $\mathbb{R}$-valued process $\eta_{t}$ that solves
an equation describing the dynamics of a tagged particle $x$ while
processes $\xi_{y,t}$, $y\in\gamma,y\neq x$, are fixed, and prove
that $\eta_{t}=\xi_{x,t}$.

\textbf{Step 4.} The uniqueness and continuous dependence on the initial
data is proved by using our version the comparison theorem and Gronwall-type
inequality (cf. \eqref{gron111}), as in \cite{ChDa}.
\end{proof}
\begin{remark}
The convergence $\Xi^{n}\rightarrow\Xi^T$ in $\mathcal{R}_{\beta}^{p}$
implies that, for any fixed $x\in\gamma^{T}$and $t\in[0,T],$ we
have $\xi_{x,t}^{n}\rightarrow\xi_{x,t}$, $n\rightarrow\infty$,
a.s.
\end{remark}

\begin{remark}\label{rem:proj}
Let $T_1<T$ and consider the natural projection $P_{T\rightarrow T_1}:{\mathbb{S}}^T\rightarrow {\mathbb{S}^{T_1}} $, induced by the embedding $\gamma^{T_1}\subset \gamma^T$. Now we can construct the process $\Xi^{T_1}$,  $t\in [0,T_1]$. The uniqueness of the solution implies that, for $t\in [0,T_1]$, we have
\begin{equation}
\Xi^{T_1}_t=P_{T\rightarrow T_1}\Xi^{T}_t.
\end{equation}
In particular, for any $t\in [0,T)$,
\begin{equation}
  \Xi^{t}_t=P_{T\rightarrow t}\Xi^{T}_t
\end{equation}
so that
\begin{equation}
    \xi_{x,t}^t=\xi_{x,t}^T,\,x\in\gamma^t.
\end{equation}
In what follows, we skip the superscript $T$ (or $t$) and just write $\xi_{x,t}$ for $x\in\gamma^t$.
\end{remark}

\begin{remark}\label{rem:gamma-dep}
    We will sometimes use the notation $\Xi_t(\gamma_t)$, to highlight the dependence of the solution $\Xi_t$ on the underlying process $\gamma_s,\ s\le t$.
\end{remark}

\begin{remark}
    Observe that, at any fixed time $t\in\left[0,T\right]$, the spin
dynamics is coupled only for the particles with positions $x\in\gamma_{t}$.
\end{remark}

\begin{remark}\label{rem:gamma-const}
 For a constant process $\gamma_t\equiv\gamma$,   the coefficients of system \eqref{MainSystem} are time-independent. Thus $\Xi_t(\gamma)$ is a homogeneous Markov process in $S^{\gamma^T}$ with generator
 $H^{\Xi(\gamma)}$ acting on cylinder functions $f\in {\mathcal{FC}}(S^{\gamma^T})$ as
 \begin{equation}\label{eq:xi-gen}
 H^{\Xi(\gamma)}f(\sigma)= \sum_{y\in \gamma}\Phi_y(\sigma)\nabla_x f(\sigma)+  \sum_{y\in \gamma}\Psi_y(\sigma)\Delta_x f(\sigma) .
 \end{equation}
\end{remark}

\section{Dynamics in $\Gamma(X,S)$.}\label{sec:dynamics}

We are now in a position to construct a process in $\Gamma(X,S)$
by combining processes $\gamma_{t}$ and $\Xi_{t}^{T}=\left(\xi_{x,t}\right)_{x\in\gamma^T}$, $t\in\left[0,T\right]$, the latter given in Theorem \ref{Existence}. We will use the fibration structure described in Remark \ref{fibre}.

\begin{theorem}\label{theor:solution}
For any $(\gamma_t)_{t\in [0,T]}\in\mathcal{R}(\mathrm{L}\Gamma(X))$, $\widehat{\gamma}_0\in \Gamma(X,S)$ and a measurable map $\varphi:X\rightarrow S$,     there exists a  unique (up to modification) adapted $\Gamma(X,S)$-valued process $(\widehat{\gamma}_{t})_{t\in[0,T]}$ on the probability space $\bf P$
with initial value   $\widehat{\gamma}_0$ and such that, a.s. for all $t\in[0,T]$,
\begin{enumerate}
\item[(1)]  $p_{X}(\widehat{\gamma}_{t})=\gamma_{t}$, and
\item[(2)]  $i_{\gamma_t}(\widehat{\gamma}_{t})=\left(\xi_{x,t}\right)_{x\in\gamma_t}$,
\end{enumerate}
where $\Xi_{t}^{T}=\left(\xi_{x,t}\right)_{x\in\gamma^{T}}$
is the solution of system \eqref{MainSystem} along $\gamma_{t}$,
with $\xi_{x,0}=\varphi(x)$ for $x\notin\gamma_0$.
The process $\widehat{\gamma}_{t},\,t\in[0,T]$,
is càdlàg.
\end{theorem}

    \begin{remark} \label{rem:ini}
The construction of process $\Xi_{t}^{T}$ requires initial condition \\
 $(\zeta_{x})_{x\in\gamma^{T}}\in \mathbb{S}^{T}$ and not just $(\zeta_{x})_{x\in\gamma_0}=i_{\gamma_0}(\widehat{\gamma}_{0})$. We consider the case where $(\zeta_{x})_{x\notin\gamma_0}$ is specified by a deterministic function $\varphi$, reflecting the properties of the environment in which the underlying process $\gamma_t$ lives. Alternatively, $(\zeta_{x})_{x\notin\gamma_0}$ can be defined
by a collection of i.i.d. random variables; another possibility would be to replace the underlying process $\gamma_t$ by a BaD process $(\gamma,\zeta)_t\in \Gamma(X,S)$, so that a new point $x$ appears together with the corresponding mark $\zeta_x$.
\end{remark}

\begin{proof}
Let us fix a realisation of the process $\gamma_t$ with
$\gamma_0=p_X(\widehat{\gamma}_0)$ and consider its phantom track $\gamma^T$. According to Theorem \ref{Existence}, there exists the unique solution $\Xi_{t}^{T}=\left(\xi_{x,t}\right)_{x\in\gamma^{T}}$
 of system \eqref{MainSystem} along $\gamma_{t}$.

 We set now $\Xi_{t}:=\left(\xi_{x}(t)\right)_{x\in\gamma_{t}}\in S^{\gamma_{t}}$
and define
 $\widehat{\gamma}_{t}$  by the formula
\begin{equation*}
\widehat{\gamma}_{t}=\left(\gamma_{t},\Xi_{t}\right)\in\Gamma(X,S),\ t\in[0,T],\label{eq:comb-proc}
\end{equation*}
cf. Remark \ref{fibre}.
Proposition \ref{thm:embed1-1} below shows that $\widehat{\gamma}_{t}$ is a random variable on $\mathbb{P}$. It is obvious that $\widehat{\gamma}_{t}$ satisfies conditions (1) and (2). The càdlàg property is shown in Proposition \ref{prop:cadlag} below. The uniqueness follows from the uniqueness of solutions of system \eqref{MainSystem}.

\end{proof}

\begin{proposition}
\label{thm:embed1-1}The map $\Omega\ni\omega\mapsto$$(\gamma_{t},\Xi_{t})\in\Gamma(X,S)$
is measurable for any $t\in[0,T]$.
\end{proposition}

\begin{proof}
\textbf{Step 1}. Consider the solution $\Xi=\Xi(\gamma_{t})$ and
the corresponding approximating sequence $\Xi^{n}=\Xi^{n}(\gamma_{t})$,
see Appendix \ref{sec:cutoff}. We will first show that, for any fixed $t\in[0,T]$,
\begin{equation}
(\gamma_{t},\Xi_{t}^{n})\rightarrow(\gamma_{t},\Xi_{t}),\ n\rightarrow\infty,\label{eq:conv1-1}
\end{equation}
a.s. in the topology of $\Gamma(X,S).$ We know that
\[
\xi_{x,t}^{n}\rightarrow\xi_{x,t},\ n\rightarrow\infty,
\]
for any $x\in\gamma,$a.s. Thus, for any function $F\in C_{0,b}(X,S),$we
have the convergence
\[
\sum_{x\in\gamma_{t}}F(x,\xi_{x,t}^{n})\rightarrow\sum_{x\in\gamma_{t}}F(x,\xi_{x,t}),\ n\rightarrow\infty,
\]
because the number of non-zero terms in both sums is finite, which
implies (\ref{eq:conv1-1}).

\textbf{Step 2}. Assume that $\gamma_{t}\equiv\gamma$ is independent
of $t$ , fix $n\in\mathbb{N}$ and prove that the map
\begin{equation}
\Gamma(X)\in\gamma\mapsto(\gamma,\Xi_{t}^{n}(\gamma))\in\Gamma(X,S)\label{eq:conv2-2}
\end{equation}
is continuous.

Consider a sequence $\left\{ \gamma^{m},m\in\mathbb{N}\right\} \in\Gamma(X)$
converging to $\gamma$ as $m\rightarrow\infty.$ It is sufficient
to show that
\begin{equation}
\sum_{x\in\gamma^{m}}F(x,\xi_{x,t}^{n}(\gamma^{m}))\rightarrow\sum_{x\in\gamma}F(x,\xi_{x,t}^{n}(\gamma)),\ m\rightarrow\infty,\label{eq:F-conv-1}
\end{equation}
for any function $F\in C_{0,b}(X,S).$ Let $B$ be a compact subset
of $X$ such that $supp\ F\subset B\times S$ and denote $B_{n}:=B\cup\Lambda_{n}$.
Then, for $m$ big enough, this sequence stabilizes in $B_{n},$that
is, there exist $M,k\in\mathbb{N}$ such that
\[
\gamma\cap B_{n}=\left\{ x_{1},...,x_{k}\right\} \ and\ \gamma^{m}\cap B_{n}=\left\{ x_{1}^{m},...,x_{k}^{m}\right\}
\]
for all $m\ge M$, and $x_{j}^{m}\rightarrow x_{j},\ m\rightarrow\infty,\ j=1,...,k$.
Therefore, $\xi_{x_{j}}^{n}(\gamma^{m}),\ j=1,...,k$, $m\ge M$,
is a solution of the system of $k$ SDEs in $S,$with the coefficients
continuously depending on $x_{1}^{m},...,x_{k}^{m}$. Thus
\[
\xi_{x_{j}^{m}}^{n}(\gamma^{m})\rightarrow\xi_{x_{j}}^{n}(\gamma),\ m\rightarrow\infty,\ j=1,...,k.
\]
Then
\[
F(x_{j}^{m},\xi_{x_{j}^{m}}^{n}(\gamma^{m}))\rightarrow F(x_{j},\xi_{x_{j}}^{n}(\gamma))\ m\rightarrow\infty,\ j=1,...,k,
\]
and (\ref{eq:F-conv-1}) holds, which implies the continuity of (\ref{eq:conv2-2})
by the definition of the topology of $\Gamma(X,S)$.

\textbf{Step 3. }We continue working with fixed $n\in\mathbb{N}$,
but now return to the general setting where  $\gamma_{t}$ is time-dependent and
show that the map
\begin{equation}
\Omega\ni\omega\mapsto(\gamma,\Xi_{t}^{n}(\gamma))\in\Gamma(X,S)\label{eq:conv-3}
\end{equation}
 is measurable. It is sufficient to show that, for any function $F\in C_{0,b}(X,S)$
and $t\in\mathcal{T},$ the map
\[
\Omega\ni\omega\mapsto\sum_{x\in\gamma_{t}}F(x,\xi_{x,t}^{n})\in\mathbb{R}
\]
is measurable. Let us fix arbitrary $t_{0}\in[0,T]$. Then there exists
$\Delta t$ such that $\gamma_{t}\cap B_{n}\equiv\gamma$ for $t\in[t_{0},t_{0}+\Delta t]\subset\left[0,T\right]$,
and $\xi_{x,t}^{n}(\gamma_{t})=\xi_{x,t}^{n}(\gamma)$ for $x\in\gamma_{t}\cap B_{n}$.
The desired measurability follows now from Step 2 combined with the
measurability of the process $\gamma_{t}\in\Gamma(X)$.

\textbf{Step 4.} The statement of the theorem follows now by the combination
of (\ref{eq:conv1-1}) and (\ref{eq:conv-3}).
\end{proof}
\smallskip{}

\begin{proposition}\label{prop:cadlag}
The process \textup{$\widehat{\gamma}_{t},\ t\in[0,T],$} is càdlàg.
\end{proposition}
\begin{proof}
Let us first show that the map $[0,T]\ni t\mapsto\widehat{\gamma}_{t}\in\Gamma(X,S)$
is right-continuous. By the definition of topology on $\Gamma(X,S)$,
the continuity in question is equivalent to the right-continuity of
the map $[0,T]\ni t\mapsto\sum_{x\in\gamma_{t}}F(x,\xi_{x,t})\in\mathbb{R}$
for every function $F\in C_{0,b}(X\times S)$ with support in $B\times S$,
$B\subset X$ compact. Observe that the number of non-zero elements
in the sum above is finite, and the map $[0,T]\ni t\mapsto\xi_{x,t}\in\mathbb{R}$
is continuous for any $x\in\gamma^{T}$. Let us fix arbitrary $t_{0}\in[0,T]$.
There exists $\Delta t$ such that $\gamma_{t}\cup B=\gamma$ for
$t\in[t_{0},t_{0}+\Delta t]\subset\left[0,T\right]$, where $\gamma\in\Gamma(X)$
is independent of $t$. Therefore, the map $[t_{0},t_{0}+\Delta t]\ni t\mapsto\sum_{x\in\gamma_{t}}F(x,\xi_{x,t})=\sum_{x\in\gamma}F(x,\xi_{x,t})\in\mathbb{R}$
is right-continuous.

\smallskip\noindent
The existence of left limits can be shown similarly.
\end{proof}

\section{Dynamics of underlying configurations}\label{Sec DoC}

Below, we consider a class of jump processes satisfying conditions
of Section \ref{sec:jump_proc}.


Throughout this section, we identify a configuration $\gamma\in\Gamma(X)$
with a discrete (counting) measure on $({{\mathbb{R}}^{d}},{\mathcal{B}}({{\mathbb{R}}^{d}}))$
defined by assigning a unit mass to each atom at $x\in\gamma$; in
particular,
\[
\gamma(\Lambda)={\cal N}(\gamma\cap\Lambda)=\sum_{x\in\gamma\cap\Lambda}\gamma(\{x\}),\qquad\Lambda\in{\mathcal{B}}_{\mathrm{b}}({{\mathbb{R}}^{d}}),\ \gamma\in\Gamma(X).
\]
We fix an arbitrary $\varepsilon>0$ and consider the function
\begin{equation}
G(x):=(1+|x|)^{-d-\varepsilon},\quad x\in{{\mathbb{R}}^{d}},\label{defG}
\end{equation}
where $|x|$ denotes the Euclidean norm on ${{\mathbb{R}}^{d}}$.
We denote then
\[
\Gamma_{G}:=\Bigl\{\gamma\in\Gamma(X)\Bigm\vert\langle G,\gamma\rangle:=\sum_{x\in\gamma}G(x)<\infty\Bigr\},
\]
a set of \emph{tempered} configurations. We define a sequential topology
on $\Gamma_{G}$ by assuming that $\gamma_{n}\to\gamma$, $n\to\infty$,
if only
\begin{equation}
\lim_{n\to\infty}\langle f,\gamma_{n}\rangle=\langle f,\gamma\rangle\label{eq:convtop}
\end{equation}
for all $f\in C_{b}({{\mathbb{R}}^{d}})$ (the space of bounded continuous
functions on ${{\mathbb{R}}^{d}}$) such that

\begin{equation}\label{eq:G-class}
|f(x)|\leq M_{f}G(x)\;
\text{for some } M_{f}>0\; \text{and all } x\in{{\mathbb{R}}^{d}}.
\end{equation}
Let ${\mathcal{B}}_{G}(\Gamma(X))$
denote the corresponding Borel $\sigma$-algebra.

\begin{remark}
    Observe that \eqref{eq:G-class} holds for any $f$ with compact support. Thus any $G$-convergent sequence $(\gamma_n)\in\Gamma_G$ converges in the vague topology, which is defined by pairing with compactly supported functions. That is, the $G$-topology is stronger than the vague topology.
\end{remark}

 We will need the following auxiliary result.
\begin{lemma}\label{lem:conv}
 For any $\alpha>0$, $k\in\mathbb{N}$ and $\gamma\in\Gamma_G(X)\cap \mathrm{L}\Gamma(X)$ we have
 \begin{equation}
  C_{\alpha,k,R}(\gamma):=\sum_{x\in\gamma}   e^{-\alpha \vert x\vert}n_{x,R}(\gamma)^k<\infty.
\end{equation}

\end{lemma}
\begin{proof}
The result immediately follows from the inequality
\begin{equation}
 e^{-\alpha \vert x\vert}(1+\log(1+  \vert x\vert))^k\le G(x),
\end{equation}
which holds for sufficiently large $\vert x\vert$.

\end{proof}

Let $b:{{\mathbb{R}}^{d}}\times\Gamma_{G}\to{\mathbb{R}}_{+}:=[0,\infty)$
be a measurable function; and let $m>0$ be a constant. We describe
a spatial birth-and-death process $\gamma:{\mathbb{R}}_{+}\to\Gamma_{G}$
with the constant death rate $m$ and the birth rate $b$ through
the following three properties:
\begin{enumerate}
\item If the system is in a state $\gamma_{t}\in\Gamma_{G}$ at the time
$t\in{\mathbb{R}}_{+}$, then the probability that a new particle
appears (a ``birth'' happens) in a $\Lambda\in{\mathcal{B}}_{\mathrm{b}}({{\mathbb{R}}^{d}})$
during a time interval $[t;t+\Delta t]$ is
\begin{equation}
\Delta t\int\limits_{\Lambda}b(x,\gamma_{t})dx+o(\Delta t).\label{meaning_of_birthrate}
\end{equation}
\item If the system is in a state $\gamma_{t}\in\Gamma_{G}$ at the time
$t\in{\mathbb{R}}_{+}$, then, for each $x\in\gamma_{t}$, the probability
that the particle at $x$ dies during a time interval $[t;t+\Delta t]$
is $m\cdot\Delta t+o(\Delta t)$.
\item With probability $1$ no two events described above happen simultaneously.
\end{enumerate}
The (heuristic) generator of our process is
\begin{equation}
LF(\gamma)=m\sum\limits_{x\in\gamma}\bigl(F(\gamma\setminus\{x\})-F(\gamma)\bigr)+\int\limits_{{\mathbb{R}}^{d}}b(x,\gamma)\bigl(F(\gamma\cup\{x\})-F(\gamma)\bigr)dx.\label{the generator}
\end{equation}

\begin{definition} \label{def:sol} Let a right-continuous complete
filtration $\{\mathcal{F}_{t}\}$ be fixed.
\begin{enumerate}
\item Let $N$ be the Poisson point process on ${\mathbb{R}}_{+}\times{\mathbb{R}}^{d}\times{\mathbb{R}}_{+}^{2}$
with mean measure $ds\times dx\times du\times e^{-r}dr$, compatible
w.r.t. the filtration $\{\mathcal{F}_{t}\}$. The latter means that
for any measurable $A\subset{\mathbb{R}}^{d}\times{\mathbb{R}}_{+}^{2}$,
$N([0,t],A)$ is $\mathcal{F}_{t}$-measurable and $N((t,s],A)$ is
independent of $\mathcal{F}_{t}$ for $0<t<s$.
\item Let $\gamma_{0}$ be a $\Gamma_{G}$-valued $\mathcal{F}_{0}$-measurable
random variable independent on $N$. Consider a point process $\tilde{\gamma}_{0}$
on ${\mathbb{R}}^{d}\times{\mathbb{R}}_{+}$ obtained by attaching
to each point of $\gamma_{0}$ an independent unit exponential random
variable. Namely, if $\gamma_{0}=\{x_{i}:i\in{\mathbb{N}}\}$ then
$\tilde{\gamma}_{0}=\{(x_{i},s_{i}):i\in{\mathbb{N}}\}$ and $\{s_{i}\}$
are independent unit exponentials, independent of $\gamma_{0}$ and
$N$.
\item We will say that a process $(\gamma_{t})_{t\geq0}$ with sample paths
in the Skorokhod space $D_{\Gamma_{G}}[0,\infty)$ has the constant
death rate $m$ and the birth rate $b$ if it is adapted to a filtration
$\{\mathcal{F}_{t}\}$ w.r.t. to which $N$ is compatible and if,
for any $\Lambda\in{\mathcal{B}}_{\mathrm{b}}({{\mathbb{R}}^{d}})$,
the following equality holds almost surely
\begin{equation}
\begin{aligned}\gamma_{t}(\Lambda) & =\int\limits_{(0,t]\times\Lambda\times\R_{+}^{2}}1\!\!1_{[0,b(x,\gamma_{s-})]}(u)1\!\!1_{(m(t-s),\infty)}(r)N(ds,dx,du,dr)\\
 & \quad+\int\limits_{\Lambda\times\R_{+}}I_{(m\,t,\infty)}(r)\tilde{\gamma}_{0}(dx,dr).
\end{aligned}
\label{seGK}
\end{equation}

\end{enumerate}
\end{definition}

\begin{remark}
\begin{enumerate}
\item The mapping ${\mathbb{R}}_{+}\times\Omega\ni(t,\omega)\mapsto\gamma_{t}(\omega)\in\Gamma$
defines a jump process, or, equivalently, $\gamma_{\cdot}(\omega)$
is a jump-function for a.a. $\omega\in\Omega$. Therefore, $\gamma_{s-}$
is well defined in \eqref{seGK} for each $s>0$, and also all moments
of random events (both of births and deaths) form a countable subset
of ${\mathbb{R}}_{+}$.
\item The integrals in \eqref{seGK} are sums. Indeed, $N$ is a (counting)
random measure on the space ${\mathbb{R}}_{+}\times{{\mathbb{R}}^{d}}\times{\mathbb{R}}_{+}^{2}$,
i.e. a random configuration $N=N(\omega)=\{(s,x,u,r)\}$, where $\omega\in\Omega$
(the underlying probability space). Then, for a.a. $\omega\in\Omega$,
\begin{equation}
\begin{aligned}\gamma_{t}(\Lambda) & =|\gamma_{t}(\omega)\cap\Lambda|=\\
 & =\sum_{(s,x,u,r)\in N(\omega)}1\!\!1_{0<s\leq t}(\omega)1\!\!1_{x\in\Lambda}1\!\!1_{u\leq b(x,\gamma_{s-})}(\omega)1\!\!1_{r>m(t-s)}(\omega)\\
 & \quad+\sum_{(x,r)\in\tilde{\gamma}_{0}(\omega)}1\!\!1_{x\in\Lambda}(\omega)1\!\!1_{r>mt}(\omega)
\end{aligned}
\label{eq:rewrite}
\end{equation}

\item
Let us use the notation $\gamma_{(b,m);t}$ for the process $\gamma_{t}$
with birth and death rates $b$ and $m$, respectively. Assume that
$\boldsymbol{b}:=\sup_{\gamma\in\Gamma(X),x\in\gamma}b(x,\gamma)<\infty$.
Then, for any $t\in[0,T]$, we have the following chain of inequalities:
\[
\gamma_{(b,m);t}(\Lambda)\le\gamma_{(\boldsymbol{b},m);t}(\Lambda)\le\gamma_{(\boldsymbol{b},0);t}(\Lambda).
\]
Let us remark that these inequalities hold because $\boldsymbol{b}$and
$m$ are independent of $\gamma$. Thus we also have
\[
\gamma_{(b,m)}^{t}(\Lambda)\le\gamma_{(\boldsymbol{b},m)}^{t}(\Lambda)\le\gamma_{(\boldsymbol{b},0)}^{t}(\Lambda).
\]
Observe now that for a process $\gamma_{t}$ with zero mortality its
phantom does not contain any additional points, so that
\[
\gamma_{(\boldsymbol{b},0)}^{t}(\Lambda)=\gamma_{(\boldsymbol{b},0);t}(\Lambda).
\]
It follows then from ($\text{\ref{eq:rewrite})}$ that, for any $t\in[0,T]$,
\begin{equation}\label{eq:rewrite1}
\begin{aligned}
    \gamma_{(b,m)}^{t}(\Lambda)\le\gamma_{(\boldsymbol{b},0);t}(\Lambda)=& N\left([0,t]\times\Lambda\times[0,{\bf b}]\times[0,\infty)\right)+{\gamma}_{0}(\Lambda)\\&
    ={\tilde{N}}\left([0,t]\times\Lambda\times[0,{\bf b}]\right)+{\gamma}_{0}(\Lambda),
\end{aligned}
\end{equation}
where $\tilde{N}$ is the Poisson point process on ${\mathbb{R}}_{+}\times{\mathbb{R}}^{d}\times{\mathbb{R}}_{+}$
with intesity measure $ds\times dx\times du$, because $e^{-r}dr$ is a probability measure on $\mathbb{R}_+$, cf. Def. \ref{def:sol}.1.

\end{enumerate}
\end{remark}

\begin{lemma}
 Assume that $\gamma_0\in {\mathrm L}\Gamma(X)$. Then $\gamma^t\in{\mathrm L}\Gamma(X)$ for any $t\in[0,T].$
\end{lemma}
\begin{proof}
   By \eqref{eq:rewrite1}, it is sufficient to estimate
   $N([0,t]\times B_{x,R}\times [0,\textbf{b}])$. Observe that, for $R$ big enough,
   $B_{x,R}\subset \tilde{B}_{(0,x,0),R}$, the ball in ${\mathbb{R}}_{+}\times{\mathbb{R}}^{d}\times{\mathbb{R}}_{+}$ centerd at $(0,x,0)$. The result follows now from the properties of Poisson point processes, see Remark \ref{rem:Poisson}.
\end{proof}

\begin{proposition} \label{prop:jumpevent} Let process $(\gamma_{t})_{t\geq0}$
satisfy conditions of Definition~\ref{def:sol}, in particular, let
\eqref{seGK} hold. Then
\begin{enumerate}
\item For each $x\in\gamma_{0}$,
\[
\gamma_{t}(\{x\})-\gamma_{t-}(\{x\})=\begin{cases}
-1, & \text{if }(x,mt)\in\tilde{\gamma_{0}},\\
0, & \text{otherwise},
\end{cases}
\]
that corresponds to the death of an initially present at $x$ point
after exponential life-time $t$ (with the parameter $m$);
\item For a.a. $x\notin\gamma_{0}$,
\begin{equation}
\begin{aligned}\gamma_{t}(\{x\})-\gamma_{t-}(\{x\}) & =\sum_{(t,x,u,r)\in N}1\!\!1_{0\leq u\leq b(x,\gamma_{t-})}\\
 & \quad-\sum_{(s,x,u,r)\in N}1\!\!1_{0<s<t}1\!\!1_{0\leq u\leq b(x,\gamma_{s-})}1\!\!1_{r=m(t-s)}.
\end{aligned}
\label{eq:diff}
\end{equation}
The first sum is $1$ iff a birth at $x$ happened at time $t$. A
summand in the second sum is $1$ iff a point born at $x$ before
time $t$ and died at time $t$ (with the life-time $t-s$).
\end{enumerate}
\end{proposition}

\begin{proof} Fix an a.a. $\omega\in\Omega$. For each $t>0$ and
$x\in{{\mathbb{R}}^{d}}$, one has, by \eqref{seGK},
\begin{align*}
\gamma_{t}(\{x\}) & =\sum_{(s,x,u,r)\in N}1\!\!1_{0<s\leq t}1\!\!1_{0\leq u\leq b(x,\gamma_{s-})}1\!\!1_{r>m(t-s)}+\sum_{(x,r)\in\tilde{\gamma}_{0}(\omega)}1\!\!1_{r>mt};\\
\gamma_{t-}(\{x\}) & =\lim_{\varepsilon\to0+}\sum_{(s,x,u,r)\in N}1\!\!1_{0<s\leq t-\varepsilon}1\!\!1_{0\leq u\leq b(x,\gamma_{s-})}1\!\!1_{r>m(t-\varepsilon-s)}\\
&+\lim_{\varepsilon\to0+}\sum_{(x,r)\in\tilde{\gamma}_{0}}1\!\!1_{r>m(t-\varepsilon)}=\\
 & =\sum_{(s,x,u,r)\in N}1\!\!1_{0<s<t}1\!\!1_{u\leq b(x,\gamma_{s-})}1\!\!1_{r\geq m(t-s)}+\sum_{(x,r)\in\tilde{\gamma}_{0}}1\!\!1_{r\geq mt}.
\end{align*}
Therefore,
\begin{align*}
\gamma_{t}(\{x\})-\gamma_{t-}(\{x\}) & =-\sum_{(s,x,u,r)\in N}1\!\!1_{0<s<t}1\!\!1_{0\leq u\leq b(x,\gamma_{s-})}1\!\!1_{r=m(t-s)}\\
 & \quad+\sum_{(s,x,u,r)\in N}1\!\!1_{s=t}1\!\!1_{0\leq u\leq b(x,\gamma_{s-})}1\!\!1_{r>m(t-s)}-\sum_{(x,r)\in\tilde{\gamma}_{0}}1\!\!1_{r=mt}\\
 & =-\sum_{(s,x,u,r)\in N}1\!\!1_{0<s<t}1\!\!1_{0\leq u\leq b(x,\gamma_{s-})}1\!\!1_{r=m(t-s)}\\
 & \quad+\sum_{(t,x,u,r)\in N}1\!\!1_{0\leq u\leq b(x,\gamma_{t-})}1\!\!1_{r>0}-1\!\!1_{(x,mt)\in\tilde{\gamma}_{0}}.
\end{align*}
Since for a.a. $\omega$, $N(\omega)$ does not contain points $(s,x,u,r)$
with fixed $r=0$, we may assume that the chosen $\omega$ is such
that $r>0$ for all $(s,x,u,r)\in N$. (Note that $r=m(t-s)$ can't
be excluded as $s$ depends on $\omega$.) Therefore,
\begin{align*}
\gamma_{t}(\{x\})-\gamma_{t-}(\{x\}) & =-\sum_{(s,x,u,r)\in N}1\!\!1_{0<s<t}1\!\!1_{0\leq u\leq b(x,\gamma_{s-})}1\!\!1_{r=m(t-s)}\\
 & \quad+\sum_{(t,x,u,r)\in N}1\!\!1_{0\leq u\leq b(x,\gamma_{t-})}-1\!\!1_{(x,mt)\in\tilde{\gamma}_{0}}.
\end{align*}
Similarly, since $\gamma_{0}$ is a fixed countable set, then, for
a.a. $\omega$, $N$ does not contain points $(s,x,u,r)$ such that
$x\in\gamma_{0}$. Recall also that $\tilde{\gamma}_{0}$ is independent
on $N$. As a result, one gets the statement. \end{proof}

\begin{remark} As a result, for a.a. $\omega\in\Omega$ and for all
$t\geq0$,
\begin{align*}
\gamma_{t}\setminus\gamma_{0}=\gamma_{t}(\omega)\setminus\gamma_{0}=\Bigl\{ x\in{{\mathbb{R}}^{d}}\bigm\vert & \exists(s,u,r)\in{\mathbb{R}}_{+}^{3}:\\
 & s\in(0,t],u\in\lbrack0,b(x,\gamma_{s-})],r>m(t-s),\\
 & (s,x,u,r)\in N(\omega)\Bigr\}.
\end{align*}
In other words, $N$, being a driving process (that is a random set,
without any time component), contains the whole `history' for $\gamma$.
The variable $s$ represents times of births in $\gamma_{t}$. Times
of deaths are not `visible' (but may be studied though), we require
instead that a point at $x$ born at time $s\in(0,t]$ does not die
within the time interval $[s,t]$ to be present at time $t$. \end{remark}

\begin{theorem}\cite[Theorem 2.1]{BDFK+2021} \label{thm_exist_uniq}
Suppose that
\begin{equation}
\mathbf{b}:=\sup_{\substack{x\in{{\mathbb{R}}^{d}}\\
\gamma\in\Gamma_{G}
}
}b(x,\gamma)<\infty,\label{b_is_bdd}
\end{equation}
and, for some $M>0$,
\begin{equation}
\sup_{\gamma\in\Gamma_{G}}\bigl\lvert b(x,\gamma\cup\{y\})-b(x,\gamma)\bigr\rvert\leq MG(x-y),\quad x,y\in{{\mathbb{R}}^{d}}.\label{main_cond_GK}
\end{equation}
Then there exists a pathwise unique solution to \eqref{seGK} in the
sense of Definition~\ref{def:sol}.
\end{theorem}

\begin{remark}\label{rem:markov}
Process $\gamma_t,\,t\in [0,T],$ is the unique solution of the martingale problem with generator \eqref{the generator}, which implies that it is Markov, see Th. 4.4.2 in \cite{EtKu}. Although the state space $\Gamma(X)$ is not separable, it is known to be standard Borel, as a subspace of all Radon measures on $X$, which is sufficient for application of the aforementioned theorem.  \end{remark}

In \cite{BDFK+2021}, conditions \eqref{b_is_bdd}--\eqref{main_cond_GK}
were checked for the following examples.

\begin{example}[{\negthinspace{}\negthinspace{}{\protect\cite[Example 4.1]{BDFK+2021}}}]
\label{ex:Glauber} We consider
\begin{equation}
b(x,\gamma)=z\exp\biggl(-\sum\limits_{y\in\gamma\setminus\{x\}}\phi(x-y)\biggr),
\end{equation}
where $z>0$ and $\phi:{\mathbb{R}}^{d}\rightarrow{\mathbb{R}}_{+}$
is such that $\phi(x)\leq BG(x)$, $x\in{\mathbb{R}}^{d}$ for some
$B>0$. Then the mapping \eqref{the generator} is the generator of
the so-called Glauber dynamics in continuum which was actively studied
in recent decades, see e.g. \cite{FKKoz2011,KL2005,FKK2011c} and
references therein. \end{example}

\begin{example}[{\negthinspace{}\negthinspace{}{\protect\cite[Theorem 3.1]{BDFK+2021}}}]
\label{ex:fecundity} We consider
\begin{equation}
b(x,\gamma)=\sum\limits_{y\in\gamma}a(x-y)\biggl(1+\sum\limits_{z\in\gamma\setminus\{y\}}c(z-y)\biggr)\exp\biggl(-\sum\limits_{z\in\gamma\setminus\{y\}}\varphi(z-y)\biggr),\label{eq:fecundity}
\end{equation}
where $a,c,\varphi:{\mathbb{R}}^{d}\to{\mathbb{R}}_{+}$ are such
that:
\begin{equation}
a(x)\leq BG^{2}(x),\qquad\varphi(x)\leq BG(x),\qquad c(x)\leq p\varphi(x)
\end{equation}
for some $B\geq1$, $p\geq0$ and a.a. $x\in{\mathbb{R}}^{d}$. This
is the model with a density dependent \emph{fecundity} rate, see \nocite{BDFK+2021,FKK2011c}
for details. \end{example}

\begin{example}[{\negthinspace{}\negthinspace{}{\protect\cite[Example~4.2]{BDFK+2021}}}]
We consider
\begin{equation}
b(x,\gamma)=\biggl(\sum\limits_{y\in\gamma}a(x-y)\biggr)\biggl(1+\sum\limits_{z\in\gamma}c(x-z)\biggr)\exp\biggl(-\sum\limits_{z\in\gamma}\phi(x-z)\biggr),\label{eq:establishment}
\end{equation}
where $a,c,\varphi:{\mathbb{R}}^{d}\rightarrow{\mathbb{R}}_{+}$ are
such that
\begin{equation}
a(x)\leq q\phi(x),\qquad c(x)\leq p\phi(x),\qquad\phi(x)\leq BG(x),\label{estkernest}
\end{equation}
for some $q,B>0$, $p\geq0$ and a.a. $x\in{\mathbb{R}}^{d}$. This
is the model with a density dependent \emph{establishment} rate, see
again \nocite{BDFK+2021,FKK2011c} for details. \end{example}

\begin{proposition} \label{prop:union_is_configuration} Let the
conditions of Theorem~\ref{thm_exist_uniq} hold. Then the process $\gamma_t,\,t\in [0,T]$, satisfies Condition \ref{cond-fin}.
\end{proposition}

\begin{proof} Let $N$ be the Poisson point process on ${\mathbb{R}}_{+}\times{\mathbb{R}}^{d}\times{\mathbb{R}}_{+}^{2}$
from Definition~\ref{def:sol}; let $(\Omega,\mathbb{P})$ be the
probability space. Then, by the Slivnyak-Mecke theorem,
for each $t\geq0$, $\Lambda\in{\mathcal{B}}_{\mathrm{b}}({{\mathbb{R}}^{d}})$,
\begin{equation}
\nu_{t,\Lambda}:=\Bigl\lvert\bigl\{(s,x,u,r)\in N\bigm\vert s\in[0,t],x\in\Lambda,u\in[0,\mathbf{b}],r\geq0\bigr\}\Bigr\rvert<\infty\quad\text{a.s.}\label{eq:finN}
\end{equation}

We define
\begin{equation}
\Pi_{t}(\Lambda):=\int\limits_{(0,t]\times\Lambda\times{\mathbb{R}}_{+}^{2}}I_{[0,\mathbf{b}]}(u)N(ds,dx,du,dr),\qquad\Lambda\in{\mathcal{B}}_{\mathrm{b}}({{\mathbb{R}}^{d}}),\ t\geq0.\label{eq:PtLa}
\end{equation}
Note that, for each $t\geq0$, $\Pi_{t}$ is the Poisson point process
on ${\mathbb{R}}^{d}$; its intensity measure is $t\;\mathbf{b}\;dx$.

Note also that $(\Pi_{t})_{t\geq0}$ obviously satisfies the conditions
of Theorem~\ref{thm_exist_uniq} with $m=0$ and $b(x,\gamma)\equiv\mathbf{b}$.
Therefore, $N$ is the driving process for both $(\gamma_{t})_{t\geq0}$
and $(\Pi_{t})_{t\geq0}$. We fix an a.a. $\omega\in\Omega$ such
that \eqref{eq:finN} holds. By \cite[Proposition~4.1 and formulas~(4.12)--(4.13)]{BDFK+2021},
we have, for the fixed $\omega$,
\begin{equation}
\gamma_{s}\setminus\gamma_{0}\subset\Pi_{s}\subset\Pi_{t},\qquad s\in[0,t];\label{eq:incl}
\end{equation}
and by \eqref{eq:finN}, \eqref{eq:PtLa},
\begin{equation}
\Pi_{t}(\Lambda)=\nu_{t,\Lambda}<\infty.\label{eq:finPi}
\end{equation}

Let $s\in(0,t]$ be a moment of birth for $\gamma_{0}\mapsto\gamma_{t}=\gamma_{t}(\omega)$
at some $x\in\Lambda$, $x\notin\gamma_{0}$. Then $(s,x,u,r)\in N=N(\omega)$
implies that $r\geq0$ and
\[
0\leq u\leq b(x,\gamma_{s-})\leq\tilde{b}.
\]
Therefore, by the analogue of \eqref{eq:diff} for $(\Pi_{t})_{t\geq0}$
(with $m=0$ and $b(x,\gamma)\equiv\mathbf{b}$),
\[
\Pi_{s}(\{x\})-\Pi_{s-}(\{x\})=1.
\]
As a result, there is a birth for $\emptyset=\Pi_{0}\mapsto\Pi_{t}$
at the point $x\in\Lambda$ at the moment of time $s$. By \eqref{eq:incl},
\eqref{eq:finPi}, the number of points in $\Pi_{s}\cap\Lambda$,
$s\in[0,t]$, increased from $0$ at $s=0$ to $\Pi_{t}(\Lambda)<\infty$.
Each increase corresponded to a birth, hence, there was a finite number
of birth for $(\Pi_{s})_{s\in[0,t]}$, and thus the same is true for
$(\gamma_{s})_{s\in[0,t]}$.

Next, since deaths in $(\gamma_{s})_{s\in[0,t]}$ may be of the existing
points only, the number of deaths in $\Lambda$ within the time interval
$(0,t]$ is bounded by the number of births then and there plus the
number $\gamma_{0}(\Lambda)$.

As a result, the mapping $[0,t]\ni s\mapsto\gamma_{s}\cap\Lambda$
changes its values only a finite number of times. \end{proof}

\bigskip

\section{Appendix }\label{sec:appendix}

In this section, we outline some results from our previous work, adapted
to the current framework. These results are relevant to the construction
of solutions of the system (\ref{MainSystem}) on a fixed typical
configuration $\gamma\equiv\gamma^{T}$.

\subsection{Finite volume cut-offs}\label{sec:cutoff}

Let us consider a compact set $\Lambda\subset X$ the following truncated
version of our original stochastic system (\ref{MainSystem}):
\begin{eqnarray}
\xi_{x,t}^{\Lambda} & = & \zeta_{x}+\int_{0}^{t}\Phi_{x}(\Xi_{s}^{\Lambda},s)ds+\int_{0}^{t}\Psi_{x}(\Xi_{s}^{\Lambda},s)dW_{x,s},\ x\in\gamma_{\Lambda}:=\gamma\cap\Lambda,\label{FinVolSystem}\\
\xi_{x,t}^{\Lambda} & = & \zeta_{x},\quad x\not\in\gamma_{\Lambda},\ \ t\in\mathcal{T},\nonumber
\end{eqnarray}
where $\bar{\zeta}=\{\zeta_{x}\}_{x\in\gamma}\in L_{\alpha}^{p}$,
$\alpha\in\mathcal{A}$, is $\mathcal{F}_{0}$-measurable random initial
condition and equality (\ref{FinVolSystem}) holds for all $t\in\mathcal{T}$,
$\mathbb{P}$-a.s. Here $\Xi_{s}^{\Lambda}:=(\xi_{x,t}^{\Lambda})_{x\in\gamma}$.
System (\ref{FinVolSystem}) is essentially finite dimensional with
only coupled equations indexed by $x\in\Lambda^{\rho}$, where $\Lambda^{\rho}$
is the $\rho$-neighborhood of $\Lambda$.
\begin{theorem}
\label{FiniteVolumeLemma} System (\ref{FinVolSystem}) admits a unique
(up to indistinguishability) solution $\Xi^{\Lambda}\in\mathcal{R}_{\alpha}^{p}$
with continuous sample paths.
\end{theorem}
\begin{proof}
The existence and uniqueness of continuous strong solutions of the
non-trivial finite dimensional part of system (\ref{FinVolSystem})
is well-known, see \cite[Chapter 3]{LiuRockner}. The inclusion $\Xi^{\Lambda}\in\mathcal{R}_{\alpha}^{p}$
follows then from the fact that $\xi_{x,t}^{\Lambda}=\zeta_{x},\ t\in\mathcal{T}$,
for $x\not\in\Lambda$.
\end{proof}
\bigskip{}

Let us now fix an expanding sequence $\{\Lambda_{n}\}_{n\in\mathbb{N}}$
of compact subsets $X$ such that $\Lambda_{n}\uparrow X$ as $n\rightarrow\infty$
and consider the sequence $\{\Xi^{n}\}_{n\in\mathbb{N}}$ of solutions
$\Xi^{n}:=\Xi^{\Lambda_{n}}$ of the corresponding systems (\ref{FinVolSystem}).
It turns out that the sequence $\{\Xi^{n}\}_{n\in\mathbb{N}}$ converges
in $\mathcal{R}_{\beta}^{p}$ for any $\beta>\alpha$.
\begin{theorem}
\label{CauchySequenceTheorem} The sequence $\{\Xi^{n}\}_{n\in\mathbb{N}}$
converges in $\mathcal{R}_{\beta}^{p}$ for any $\beta>\alpha$.
\end{theorem}
\begin{proof}
Similar to the proof of Theorem III.3. in \cite{ChDa}. The convergence
of the sequence $\{\Xi^{n}\}_{n\in\mathbb{N}}$ can be proved using
the uniform bound on the components of $\Xi^{n}$ similar to \eqref{ineq222-1},
which in turn requires an application of the generalised Gronwall
inequality, see Lemma \eqref{gron111} below.
\end{proof}

\subsection{Gronwall inequality in the scale of spaces of sequences}

Here we summarise some results from \cite{ChDa} adapted to our situation.
\begin{lemma}
\label{OvsMapTheorem}Consider an arbitrary configuration $\gamma\in\Gamma(X)$
and assume that $\{Q_{x,y}\}_{x,y\in\gamma}$ is such that for all
$x,y\in\gamma$ we have
\end{lemma}
\begin{itemize}
\item $Q_{x,y}=0$ if $\left\vert x-y\right\vert >\rho$;
\item there exist $C>0$ and $k\geq1$ such that
\begin{equation}
|Q_{x,y}|\leq Cn_{x}^{k}.\label{OvsMapTheoremEqn1}
\end{equation}
\end{itemize}
Then $Q\in\mathcal{O}(\mathcal{L}^{1},q)$ for any $q<1$, that is,
\begin{equation}
\Vert Qz\Vert_{\beta}\leq\frac{L}{(\beta-\alpha)^{q}}\Vert z\Vert_{\alpha}\label{OvsMapTheoreme1}
\end{equation}
for any $\alpha<\beta$, with
\begin{equation}
L=L(\gamma)=Ce^{\alpha^{\ast}\rho}\left[\left(\rho^{q}+n_{0,R}(\gamma)\right)(\alpha^{\ast}-\alpha_{\ast})^{q}+\left(e^{-1}q\right)^{q}\right],\label{eq:L}
\end{equation}
where
\[
n_{0,R}(\gamma)=N(\gamma\cap B_{R}).
\]
Here $B_{R}\subset X$ is the ball of radius $R$ centered at $0$,
with $R>0$ such that $n_{x}\le\left|x\right|^{q/2k}$ if $\left|x\right|>R$.
\begin{proof}
The proof is based on a direct calculation of the norm of operator
$Q:l_{\alpha}^{p}\rightarrow l_{\beta}^{P}$ and is quite tedious,
see \cite{ChDa}, Theorem A.5. The explicit expression for constant
$L$ is extracted from the proof of that theorem. Observe that the
constant $R$ as above always exists.
\end{proof}
\smallskip{}

The next result plays the key role in the proof of the convergence
of the sequence of cut-off solutions $\Xi^{n}$ and can be considered
as a generalized Gronwall inequality in the scale $\mathcal{L}^{p}$.
\begin{lemma}
\label{gron111}(\cite{ChDa}, Lemma A.8.) Consider an arbitrary configuration
$\gamma\in\Gamma(X)$ and a bounded measurable map $\varrho:\mathcal{T\rightarrow}l_{\alpha}^{1}$,
$\alpha\in\mathcal{A}$, and assume that its components satisfy the
inequality
\begin{equation}
\varrho_{x}(t)\leq Bn_{x}^{k}\sum_{y\in\bar{\gamma}_{x}}\int_{0}^{t}\varrho_{y}(s)ds+b_{x},\ t\in\mathcal{T},\ x\in\gamma,\label{ineq111}
\end{equation}
for some constants $B>0$ and $k\geq1$ and $b:=(b_{x})_{x\in\gamma}\in$
$l_{\alpha}^{1}$, $b_{x}\geq0$. Then we have the estimate
\begin{equation}
\sum_{x\in\gamma}e^{-\beta|x|}\sup_{t\in\mathcal{T}}\varrho_{x}(t)\leq K_{T}(\alpha,\beta)\sum_{x\in\gamma}e^{-\alpha|x|}b_{x}\label{ineq222}
\end{equation}
for any $\beta>\alpha$, with
\begin{equation}
K_{T}(\alpha,\beta)=\sum_{n=0}^{\infty}\frac{L^{n}T^{n}}{(\beta-\alpha)^{qn}}\frac{n^{qn}}{n!}<\infty\label{eq:K}
\end{equation}
and $L$ as in (\ref{eq:L}).
\end{lemma}
\begin{proof}
Inequality (\ref{ineq111}) can be rewritten in the form
\[
\varrho_{x}(t)\leq\sum_{y\in\bar{\gamma}}Q_{x,y}\int_{0}^{t}\varrho_{y}(s)ds+b_{x},\ t\in\mathcal{T},
\]
where
\[
Q_{x,y}=\begin{cases}
Bn_{x}^{k}, & |x-y|\leq\rho,\\
0, & |x-y|>\rho
\end{cases}
\]
for all $x\in\gamma.$ We have $\varrho\in\mathcal{B}(\mathcal{T},l_{\alpha}^{1})$,
and $|Q_{x,y}|\leq Bn_{x}^{k}$. Therefore using Theorem \ref{OvsMapTheorem}
we conclude that for any $q\in(0,1)$ matrix $\left(Q_{x,y}\right)$
generates an Ovsjannikov operator of order $q$ on $\mathcal{L}^{1}$.
We can now use Corollary A.7 from \cite{ChDa} to conclude that (\ref{ineq222})
holds. \end{proof}

\subsection{Estimates of the solution}

Now we will apply Lemma \ref{gron111} to solutions of equation (\ref{MainSystem}).
 Let us fix a realisation of the process $\gamma_t,\,t\in\mathcal{T}$, and consider its fantom $\gamma^T$.

\begin{lemma}
\label{two-bound} Let $p\geq2$ be fixed and assume that $\mathbb{R}$-valued
processes $\xi_{x,t},\,x\in\gamma^T$, satisfy equation (\ref{MainSystem}).
Then there exist universal constants $B,C_{1}$and
$C_{2}$ such that
\begin{equation}
\mathbb{E}|\xi_{x,t}|^{p}\leq\mathbb{E}\left\vert \xi_{x,0}\right\vert ^{p}+C_{1}n_{x}^{2}(\gamma^T)\sum_{y\in\bar{\gamma}_{x}^T}\int_{0}^{t}\mathbb{E}|\xi_{y,s}|^{p}ds+C_{2}n_{x}^{2}(\gamma^T), \,  x\in\gamma^{T},\label{single1}
\end{equation}
and for all$\ t\in\mathcal{T}$. The constants $B,C_{1}$ and $C_{2}$
are independent of the realization of $\gamma^{T}$ and processes $\xi_{x,t},\,x\in\gamma$ and are determined only by the constants in Condition \ref{cond-diss}.
\end{lemma}
\begin{proof}
The case of time-independent $\gamma_{t}\equiv\gamma$ has been proved
in \cite{ChDa}, Lemma A.10. In our case of time-dependent $\gamma_{t}$,
a rather standard application of the Ito formula and estimates  \eqref{DriftLemma}
show that
\begin{equation}
\mathbb{E}|\xi_{x,t}|^{p}\leq\mathbb{E}\left\vert \xi_{x,0}\right\vert ^{p}+C_{1}\int_{0}^{t}n_{x}^{2}(\gamma_{s})\sum_{y\in\bar{\gamma}_{x,s}}\mathbb{E}|\xi_{y,s}|^{p}ds+C_{2}n^2_{x}(\gamma_t),\label{single1-1}
\end{equation}
which, together with Condition \ref{cond-log-1}, implies the result.
\end{proof}
\begin{remark}
    In particular, $\bar{C}_{2}(\gamma^{T})\coloneqq\{C_{2}^{x}(\gamma^{T}):=C_{2}n_{x}^{2}(\gamma^T)\}_{x\in\gamma^{T}}\in l_{\alpha}^{p}(\gamma^{T})$.
\end{remark}
\begin{corollary}
Applying Lemma \ref{gron111} to inequality (\ref{single1}) on configuration
$\gamma^{T}$, we obtain the following growth estimate of solutions
of equation (\ref{MainSystem}):
\begin{equation}
\sum_{x\in\gamma^{T}}e^{-\beta|x|}\sup_{t\in\mathcal{T}}\mathbb{E}|\xi_{x,t}|^{p}\leq C_{1}K_{T}(\alpha,\beta;\gamma^T)\sum_{x\in\gamma^{T}}e^{-\alpha|x|}\left(\mathbb{E}\left\vert \xi_{x,0}\right\vert +C_{2}n_{x}^{2}(\gamma^{T})\right).\label{ineq222-1}
\end{equation}
In particular,
\begin{equation}
\mathbb{E}\left\Vert \Xi_{t}\right\Vert _{l_{\beta}^{p}(\gamma_{t})}^{p}\leq C_{1}K_{T}(\alpha,\beta;\gamma^T)\left(\mathbb{E}\left\Vert \Xi_{0}\right\Vert _{l_{\alpha}^{p}(\gamma^{T})}^{p}+\left\Vert \bar{C}_{2}\right\Vert _{l_{\alpha}^{p}(\gamma^{T})}\right).\label{ineq222-1-1}
\end{equation}
for any $\ t\in\mathcal{T}$.
\end{corollary}

\begin{remark}
    In the r.h.s. of the above inequality, $\gamma^T$ can be replaced by $\gamma^t$, cf. Remark \ref{rem:proj}.
\end{remark}

\end{document}